\documentclass{amsart}
\usepackage[leqno]{amsmath}

\usepackage{amsmath}
\usepackage{amsthm}
\usepackage{amssymb}
\usepackage{amscd}
\usepackage{amsxtra}     
\usepackage{epsfig}
\usepackage{verbatim}
\usepackage[all,cmtip]{xy}
\usepackage{mathrsfs}

\theoremstyle{plain} 

\newtheorem{thm}{Theorem}[section]
\newtheorem{lem}[thm]{Lemma}
\newtheorem{prop}[thm]{Proposition}

\newtheorem{cor}[thm]{Corollary}
\newtheorem{dfn}[thm]{Definition}

\newtheorem{conj}[thm]{Conjecture}

\theoremstyle{definition}

\newtheorem*{rmk}{Remark}

\numberwithin{equation}{section}

\newcommand{\Div}{\textup{div}}

\newcommand{\tensor}{\otimes}

\DeclareMathOperator{\Pic}{Pic}

 \DeclareMathOperator{\Tor}{Tor}
 \DeclareMathOperator{\Ext}{Ext}
 \DeclareMathOperator{\Hom}{Hom}

 \DeclareMathOperator{\Supp}{Supp}
 \DeclareMathOperator{\Spec}{Spec}
 
 \DeclareMathOperator{\Se}{S}

 \DeclareMathOperator{\CH}{CH}

 \DeclareMathOperator{\pd}{pd}
 \DeclareMathOperator{\height}{height}

 \DeclareMathOperator{\depth}{depth}

\DeclareMathOperator{\Cl}{Cl}

 \DeclareMathOperator{\LC}{H}

\DeclareMathOperator{\modu}{mod}
  \DeclareMathOperator{\MCM}{MCM} 
 
 \newcommand*{\SHom}{\mathcal{H}\mathit{om}}

\newcommand{\ses}[3]{0 \to {#1} \to {#2} \to {#3} \to 0}
 \newcommand{\h}[3]{\theta^{#1}({#2},{#3})}

\begin{document}

\bibliographystyle{plain}

\title[Picard groups of punctured spectra ]{Picard groups of punctured spectra of dimension three local hypersurfaces are torsion-free}

\author{Hailong Dao}
\address{Department of Mathematics\\
University of Kansas\\
 Lawrence, KS 66045-7523 USA}
\email{hdao@math.ku.edu}

\thanks{The author is partially supported by NSF grant DMS 0834050}



\maketitle

\begin{abstract}
Let $(R,m)$ be a Noetherian local ring and $U_R=\Spec(R) -\{m\}$ be the punctured spectrum of $R$. Gabber conjectured that if 
$R$ is a complete intersection of dimension $3$, then the abelian group $\Pic(U_R)$ is torsion-free. In this note we prove Gabber's statement for the hypersurface case. We also point out certain connections between Gabber's Conjecture, Van den Bergh's notion of non-commutative crepant resolutions and some well-studied questions in homological algebra over local rings. 
\end{abstract}

\section{Introduction}
Let $(R,m)$ be a  local ring (always Noetherian in this note). Let $U_R=\Spec(R) -\{m\}$ be the punctured spectrum of $R$. 
In \cite{Gab} Gabber made the following:
\begin{conj}\label{GabberConj}
Let $R$ be a local complete intersection of dimension $3$. Then $\Pic(U_R)$ is torsion-free. 
\end{conj}

The above Conjecture  is equivalent to the statement that the local flat cohomology group $H_{\{m\}}^2(\Spec(R), \mu_n)=0$
when $R$ is a local complete intersection of dimension $3$, and they are both implied by (for more details, see \cite{Gab}):

\begin{conj}\label{Brauer}
Let $R$ be a strictly henselian local complete intersection of dimension at least $4$. Then the cohomological Brauer group of $U_R$ vanishes: $Br(U_R)=0$.
\end{conj}

Conjecture \ref{GabberConj} is known when $R$ contains a field; the characteristic $0$ case follows from Grothendieck's techniques on local Leftschetz  theorems (cf. \cite{Ba, Rob}), and the positive characteristic case can be found in \cite{DLM} (it is probably known to experts, though we can not find an exact reference. It was claimed in \cite{Gab} that Conjecture \ref{Brauer} is known in positive characteristic). We also note that when  $U_R$ is replaced by a smooth projective complete intersection the analogous result on the Picard group is contained in \cite[Theorem 1.8]{De}. In any case, the main difficulty is when $R$ is of mixed characteristic.

In this paper we give a short and relatively self-contained proof of Gabber's Conjecture \ref{GabberConj} for the case of hypersurfaces, that is, if $\hat{R} \cong T/(f)$ where $T$ is a complete regular local ring. In fact, in this situation we shall prove a stronger result which is a pure commutative algebra statement. To state such result let us recall a useful notion. For a Noetherian ring $R$ one can define a map $c_1: G(R) \to \CH^1(R)$ from the Grothendieck group of finitely generated modules over $R$ to the height one component of the  Chow group of $\Spec(R)$ (see Subsection \ref{ChowPicard} for more details). Given an $R$-module $M$, we shall abuse notation a  bit and call $c_1([M])$ the first local Chern class of $M$.  Then our main result says:

\begin{thm}\label{mainTheorem}
Let $R$ be local hypersurface of dimension $3$. Let $N$ be a finitely generated reflexive $R$-module which is  locally free   on $U_R$. Furthermore, assume that the first local Chern class of $N$ is  torsion in $\CH^1(R)$. Then $\Hom_R(N,N)$ is a maximal Cohen-Macaulay $R$-module if and only if $N$ is free.  
\end{thm}

It is not hard to see that the above Theorem implies Conjecture \ref{GabberConj} in the hypersurfaces case, by taking $N$ to be the $R$-module generated by the sections of a torsion element in $\Pic(U_R)$; see section \ref{notations} and the proof of \ref{Gabber} for more details. 

This project actually arises from our  attempt to understand a striking definition by Van den Bergh of non-commutative crepant resolutions of a Gorenstein local ring $R$. To explain the connection we recall: 

\begin{dfn}(Van den Bergh, \cite{V1})
Suppose that there exists a reflexive module $N$ satisfying:
\begin{enumerate}
\item $A = \Hom_R(N,N)$ is a maximal Cohen-Macaulay $R$-module.
\item $A$ has finite global dimension equal to $d=\dim R$.
\end{enumerate} 
Then $A$ is called a non-commutative crepant resolution (henceforth NCCR) of $R$.
\end{dfn}

In \cite{Da3} we proved that non-commutative crepant resolutions can not exist when $R$ is a dimension $3$, equicharacteristic  or unramified  hypersurface with isolated singularity and torsion class group. Theorem \ref{mainTheorem} implies: 

\begin{cor}
Let $R$ be a dimension $3$ hypersurface which has isolated singularity and torsion class group (which in this case is equivalent to $R$ being an unique factorization domain by our main results). Then $R$ has no non-commutative crepant resolution in the sense of Van den Bergh.  
\end{cor}

We now briefly describe the organization of the paper. Section \ref{notations} deals with preliminary materials. In Section \ref{mainTheorems} we give the proofs of the main results announced above as well as some other interesting applications. Finally, in Section \ref{open} we raise some open questions relevant to our approach to Gabber's conjecture. 
 
{\bf Acknowledgements.} We sincerely thank two anonymous referees whose detailed comments corrected several statements and improved the paper significantly.
 
\section{Notations and preliminary results}\label{notations}
Throughout the note $R$ will be a Noetherian local ring. Recall that a maximal Cohen-Macaulay $R$-module $M$ is a finitely generated module satisfying $\depth M = \dim R$.  

 Let $\modu(R)$ and $\MCM(R)$ be the category of finitely generated  and finitely generated  maximal Cohen-Macaulay $R$-modules, respectively. Suppose  $X$ is a Noetherian scheme. Let $\mathfrak{Coh} (X)$ denote the category of coherent sheaves on $X$ and $\mathfrak{Vect}(X)$ the subcategory of vector bundles on $X$. By $G(X), \Pic(X), \CH^i(X), \Cl(X)$ we shall denote the Grothendieck group of coherent sheaves on $X$, the Picard group of invertible sheaves on $X$, the Chow group of codimension $i$ irreducible, closed subschemes of $X$, and the class group of $X$, respectively. When $X=\Spec R$ we shall write $G(R), \Pic(R), \CH^i(R), \Cl(R)$.  Let $\overline G(R):=G(R)/\mathbb Z[R]$ be the reduced Grothendieck group and $\overline G(R)_{\mathbb Q}:= \overline G(R)\tensor_{\mathbb Z}\mathbb Q$ be the reduced Grothendieck group of $R$ with rational coefficients. 

\subsection{Vector bundles on $U_R$ and modules over $R$}\label{Horrocks}
Let $\Gamma_X$ be the section functor on $X$. We have the following:

\begin{prop}(Horrocks \cite[section 1]{Hor})
Let $R$ be a Noetherian local ring such that $\depth R\geq 2$. Let $X= U_R$. Then $\Gamma_X$ induces an equivalence of category between  $\mathfrak{Vect} (X)$ and the subcategory of $\modu(R)$  consisting of finitely generated modules $M$ which is locally free on non-maximal primes with $\depth M\geq 2$ (Note that the condition $\depth R\geq 2$ also ensures that $X$ is connected).  
\end{prop}

In particular, let $\mathcal E$ represent an element in $\Pic(X)$ and let $I = \Gamma_X(\mathcal E)$. We know that $I$ is a reflexive ideal in $R$ which is locally free of rank $1$ on $X$. Furthermore 
$\Hom_R(I,I) \cong \Gamma_X(\SHom_{\mathcal O_X}(\mathcal E,\mathcal E)) \cong \Gamma_X(\mathcal O_X)= R$. 

\subsection{Some maps between Chow, Picard, and Grothendieck groups} \label{ChowPicard}
In this Subsection we assume that $R$ is a local ring such that $\depth R\geq 2$.
For $i=0,1$ there are maps $c_i: G(R) \to \CH^i(R)$.  These maps admit a very elementary definition as follows: suppose $M$ is an $R$-module. Pick any prime filtration $\mathcal F$ of $M$. Then one can take $c_i([M])  = \sum [R/p]$, where $p$ runs over all prime ideals such  that  $R/p$ appears in $\mathcal F$ and $\height(p)=i$, note that a prime can occur multiple times in the sum (for a proof that this is well-defined see the main Theorem of \cite{Ch}). When $R$ is a normal, algebra essentially of finite type over a field and $N$ is locally free (i.e. a vector bundle) on $U_R$, $c_1$ agrees with the first  Chern class of $N$, as defined in \cite[Chapter 3]{Fu}, but we shall not need that fact.  

One has the following diagram of maps of abelian groups:
\[
\xymatrix{                  &\Pic(U_R)  \ar[d]^p \\
 G(R) \ar[r]^{c_1}   & \CH^1(R)}
\]

Here $p$ is induced by the well-known map between Cartier and Weil divisors (see Chapter 2 of \cite{Fu}). 

Note that we do not indicate any map between $\Pic(U_R)$ and $G(R)$. However, the diagram ``commutes" in a 
weak sense: if $\mathcal E$ represents an element in $\Pic(X)$ and  $I = \Gamma_X(\mathcal E)$ then $p([\mathcal E]) = c_1([I])$ in $\CH^1(R)$.

Obviously, $c_1([R])=0$, so $c_1$ induces a map $q: \overline G(R) \to \CH^1(R)$. In particular, if $M$ is a module such that $[M]=0$ in $ \overline G(R)_{\mathbb Q}$ then $c_1([M])$ is torsion in $\CH^1(R)$. 

\subsection{Maximal Cohen-Macaulay approximations}\label{appro}
The reference for this Subsection is  the paper \cite{AB}.
Suppose that $R$ is Cohen-Macaulay and a homomorphic image of a Gorenstein ring. For any $R$-module $N$ there exists a short exact sequence: 
\begin{equation}
\ses WMN
\end{equation}
such that  $M\in \MCM(R)$ and $W$ has finite injective dimension. Note that if $R$ is Gorenstein, then $\pd_RW<\infty$. Also, if $R$ is Gorenstein and  $\depth N\geq \dim R-1$, then by counting depth and  the Auslander-Buchsbaum formula $W$ must be free. 

\subsection{Hochster's theta function}
Let $R $ be a local hypersurface, so $\hat R=T/(f)$ where $T$ is a regular local ring. Suppose that $M$ is an $R$ module such that $\pd_{R_p}M_p<\infty$ for any $p\in U_R$. Then for any $R$-module $N$, $\ell(\Tor_i^R(M,N))<\infty$ for $i\gg 0$; here $\ell(-)$ denotes length. 
The function $\theta^R(M,N)$ was introduced by Hochster (\cite{Ho1}) to be:
$$ \theta^R(M,N) = \ell(\Tor_{2e+2}^R(M,N)) -\ell(\Tor_{2e+1}^R(M,N)) $$
where $e$ is any integer such that $2e \geq \dim R$. It is well known (see \cite{Ei}) that the sequence of modules $\{\Tor_i^R(M,N)\}$ is periodic of
period 2 for $i>\depth R- \depth M$, so this function is well-defined. The theta function satisfies the following properties:

\begin{prop}\label{decent}(Hochster, \cite{Ho1})

\begin{enumerate}
\item If $M\tensor_RN$ has finite length, then:
$$\theta^R(M,N) = \chi^T(M,N):= \sum_{i\geq 0} (-1)^i\ell(\Tor_i^T(M,N))$$
Here $\chi^T$ is the well-known  Serre's intersection multiplicity. In particular, if  $\dim M +\dim N \leq \dim R =\dim T-1$, then $\h RMN=0$ (note that vanishing for $\chi^T$ is proved for all regular local rings; see \cite[13.1]{Ro})
\item $\theta^R(M,N)$ is bi-additive on short exact sequences,
assuming it is defined on all pairs. In particular, if $M$ is locally of finite projective dimension on $U_R$, then the rule:
$ [N] \mapsto \theta^R(M,N) $ induces maps $\overline G(R) \to \mathbb Z$ and $\overline G(R)_{\mathbb Q} \to \mathbb Q$.
\end{enumerate}
\end{prop}

The following elementary but useful result will be used in the proof of our main Theorem:

\begin{lem}(Lemma 2.3, \cite{Da3})\label{useful}
Let $R$ be a Cohen-Macaulay local ring, $M,N$ finitely generated $R$-modules and $n>1$ an integer. Consider the two conditions:
\begin{enumerate}
\item  $\Hom(M,N)$  satisfies Serre's condition $(\Se_{n+1})$.
\item  $\Ext_R^i(M,N)=0$ for $1 \leq i \leq n-1$.
\end{enumerate}
If  $M$ is locally free in codimension $n$ and $N$ satisfies $(\Se_n)$, then (1) implies (2). If  $N$ satisfies $(\Se_{n+1})$, then (2) implies (1).
\end{lem}

Finally we shall need a refined version of the  Bourbaki sequence for a module: 

\begin{thm}(\cite[Theorem 1.4]{HJW})\label{moving}
Let $R$ be a commutative, Noetherian ring satisfying condition $(\Se_2)$. Let $M$ be a torsion-free $R$-module and $S$ be a finite set of prime ideals of $R$. Assume that $M$ is free of constant rank on $S$ and the set of height at most $1$ primes in $R$. Then there is  a Bourbaki sequence $\ses FMI$ such that $I\nsubseteq \bigcup_{P\in S}P$.
\end{thm}

\section{Main results}\label{mainTheorems}

Throughout this Section $R$ will be a local hypersurface of dimension $3$. All modules are finitely generated. Note that since $\depth R > 2$, $U_R$ is connected, so any module which is locally free on $U_R$ also has constant rank.

\begin{prop}\label{theta_vanish}
Let $M$ be reflexive $R$-module which is  locally free of constant rank  on $U_R$. Let $N$ be an  $R$-module which is locally free of constant rank on the minimal primes of $R$ and such that $c_1([N])$ is torsion in $\CH^1(R)$. Then $\h RMN=0$.
\end{prop}

\begin{proof}

Without loss of generality one can assume $c_1([N])=0$ by replacing $N$ with a direct sum of copies of $N$ if necessary. 
First we claim that  in $\overline G(R)_{\mathbb Q}$, the reduced Grothendieck group with rational coefficients, we have an equality $[N] = \sum a_i[R/P_i]$ such that each  $P_i \in \Spec R$ has height at least $2$. Since $N$ has constant rank $a$ we have a short exact sequence:

$$ \ses {R^a}N{N'} $$
where  $N'$ is a torsion module. Let $\mathcal F$ be a prime filtration of $N'$. Clearly $\mathcal F$ involves only primes of height at least $1$. Let $s$ be the formal sum of all height $1$ primes in $\mathcal F$. Since $c_1([N']=c_1([N])=0$ we have formally (see Subsection \ref{ChowPicard}):
$$ s = \sum n_j\Div(f_j, R/q_j)$$ 
here the $n_j$ are integers, each $q_j$ is a minimal prime of $R$ and $f_j$ is a regular element in $R/q_j$ and by definition:
$$\Div(f_j, R/q_j)= \sum \ell(R/(q_j,f_j)_p)[R/p] $$  
(the sum runs over all primes of height $1$ in $\Supp(R/(q_j,f_j)$). The above formal equality shows that in $G(R)$ one has:
$$[N'] = \sum n_j[R/(q_j,f_j)] + \sum a_i[R/p_i] $$
such that all the primes $p_i$ are of height at least $2$ and the $a_i$ are integers. But the exact sequence $\ses {R/q_j}{R/q_j}{R/(f_j,q_j)}$ shows that each $[R/(q_j,f_j)]=0$ in $G(R)$, so our claim follows.    

Because of  the claim above we will be done by showing that $\h RM{R/P}=0$ for each $P \in \Spec R$ such that $\height P\geq 2$. 

By  Theorem \ref{moving} one can  construct a Bourbaki sequence for $M$:
$$\ses FMI $$
such that $I \subsetneq P$.  Obviously $\h RM{R/P} =\h RI{R/P}$.  But $R/I\tensor_RR/P$ has finite length, and $\dim R/I +\dim R/P \leq 3 =\dim R $. By \ref{decent} $\h R{R/I}{R/P}=0$. Since $\h RI{R/P}=-\h R{R/I}{R/P}$ we are done. 

\end{proof}

\begin{prop}\label{mainProp}
Let $M\in \MCM(R)$  such that $M$ is locally free on $U_R$ and $N$ be any finitely generated $R$-module. Suppose that  $\h R{M^*}N=0$. If   $\Ext^1_R(M,N)=0$ then $M$ is free or $\pd_RN<\infty$.
\end{prop}

\begin{proof}
One has the following short exact sequence (see \cite[3.6]{Ha} or \cite{Jo}, \cite{Jor}):
$$ \Tor_2^R(M_1,N) \to \Ext^1_R(M,R)\tensor_R N \to \Ext^1_R(M,N) \to \Tor_1^R(M_1,N) \to 0 $$
Here $M_1$ is the cokernel of $F_1^* \to F_2^*$, where ${\bf F}: \cdots \to F_2 \to F_1\to F_0 \to M \to 0$ is a minimal resolution of $M$.  Since $\Ext^1_R(M,N)=0$ it follows that $\Tor^1_R(M_1,N) =0$.

Since $M$ is $\MCM$ and $R$ is a hypersurface  we know that the minimal resolution $\bf F$ is periodic of period at most $2$ (see \cite{Ei} and $\Ext_R^i(M,R)=0$ for $i>0$. It follows that the dual complex $F^*$ is also exact and periodic of period at most $2$. Thus  $M_1$ is  isomorphic to the first syzygy of $M^*$. In particular, $M_1$ is maximal Cohen-Macaulay or zero. Since 
$\h R{M_1}N=-\h R{M^*}N=0$, it now follows that $\Tor_i^R(M_1,N)=0$ for all $i>0$ (as $M_1$ is maximal Cohen-Macaulay, the sequence of modules $\{\Tor_i^R(M_1,N)\}$ is periodic of period $2$ for $i>0$). 
So either $M_1$ or $N$ has finite projective dimension by \cite[Theorem 1.9]{HW2} or \cite[1.1]{Mil}. But if $M_1$ has finite projective dimension and is non-zero, it must be free by the Auslander-Buchsbaum formula, contradicting the minimality of ${\bf F}$. Thus $M_1$ is zero and $M$ must be free.

\end{proof}

\begin{cor}\label{mainCor}
Let $(R, \mathfrak m)$ be local hypersurface of dimension $3$. Let $N$ be a reflexive $R$-module which is  locally free   on $U_R$. Assume that $\h R{N^*}N=0$. Then $\Hom_R(N,N) \in \MCM(R)$ if and only if $N$ is free.  

\end{cor}

\begin{proof}
The sufficient direction is trivial. Suppose that $\Hom_R(N,N)$ is maximal Cohen-Macaulay. Then by Lemma \ref{useful} $\Ext^1_R(N,N)=0$. 
We look at the MCM approximation of $N$ as in \ref{appro}:
$$\ses WMN $$ 
As the discussion in  \ref{appro} indicates, $W$ is free. Applying $\Hom_R(-,N)$ we obtain $\Ext^1_R(M,N)=0$. 
Also, applying $\Hom_R(-,R )$ yields:

$$0 \to N^* \to M^* \to W^* \to \Ext_R^1(N,R) \to 0 $$
since $\Ext_R^1(M,R)=0$ as its Matlis dual  is $\LC_{\mathfrak m}^2(M)=0$.  Note that $\h RL{-}$ is always defined if $L$ is any of the four modules in the above exact sequence, and it is $0$ when $L=W^*$  or $L= \Ext_R^1(N,R)$ (the latter is because $\Ext_R^1(N,R)$  has finite length and Proposition \ref{decent}). So $\h R{M^*}N = \h R{N^*}N=0$.  

Proposition \ref{mainProp} shows that either $M$ is free or $\pd_RN<\infty$.  
Both possibilities imply that $ \pd_RN<\infty$. As $N$ is reflexive and $\dim R=3$, $\pd_RN\leq 1$. But $\Ext^1_R(N,N)=0$, so $\pd_RN$ can not be $1$ by Nakayama's Lemma, thus $N$ is free. 
\end{proof}

We have gathered enough to prove our main result:
\begin{thm}\label{main}
Let $R$ be local hypersurface of dimension $3$. Let $N$ be a reflexive $R$-module which is  locally free   on $U_R$. Furthermore, assume that the image $c_1([N])$  (of $N$ as an element in $G(R)$) is torsion in $\CH^1(R)$. Then $\Hom_R(N,N) \in \MCM(R)$ if and only if $N$ is free.  
\end{thm}

\begin{proof}
A combination of Proposition \ref{theta_vanish} and Corollary \ref{mainCor} give the desired result.  
\end{proof}

\begin{cor}\label{Gabber}
Let $R$ be local hypersurface of dimension $3$. Then $\Pic{U_R}$ is torsion-free. 
\end{cor}

\begin{proof} 
Let $\mathcal E$ represent a  torsion element in $\Pic{U_R}$. By \ref{Horrocks}  $I= \Gamma_X(\mathcal E)$ is  a reflexive ideal which is locally free of rank $1$ on $U_R$.  By the diagram in Subsection \ref{ChowPicard} we know that  $c_1([I])$ is torsion in $\CH^1(R)$. Theorem \ref{main} now applies directly to give the desired result.

\end{proof}

Finally we note some interesting consequences of the main results above:

\begin{thm}\label{UFD}
Let $R$ be a local hypersurface with isolated singularity and $\dim R=3$. The following are equivalent:

\begin{enumerate}

\item $\h RMN=0$ for all $M,N \in\modu(R)$.
\item $R$ is a unique factorization domain (equivalently, $\CH^1(R)=\Cl(R)=0$). 
\item The class group $\Cl(R)$ is torsion.

\end{enumerate}
\end{thm}

\begin{proof}
First, since $R$ is local and normal (by Serre's criterion), it is a domain (see \cite[Theorem 23.8]{mat}). Assume $(1)$.  Let $I$ be a reflexive ideal representing an element of $\Cl(R)$. Then $\Hom_R(I,I) \cong R$, and  Corollary \ref{mainCor} implies $I$ is principal, so $\Cl(R)=0$. 

The implication $(2) \Rightarrow (1)$  follows from \ref{theta_vanish}. 
The equivalence $(2) \Leftrightarrow (3)$ is implied by main Theorem \ref{Gabber}.

\end{proof}

\begin{rmk}
If $\hat R$ is a hypersurface in an equicharacteristic or unramified regular local ring then the above result follows from \cite[Corollary 3.5]{Da1} and \cite[Theorem 3.4]{Da2}. We also note that the equivalence $(1)\Leftrightarrow (3)$ when $k=\mathbb C$ and $R$ is graded is obtained in \cite[3.10, 6.1]{MPSW}.
\end{rmk}

\section{Open questions}\label{open}

In this section we discuss some open questions motivated by the results obtained previously. Clearly, the most important question is whether or not Theorem \ref{mainTheorem} is true when $R$ is a local complete intersection of dimension $3$. Affirmation of such a statement would immediately prove Gabber's Conjecture \ref{GabberConj}. 
Since for an $R$-module $M$, $c_1([M])$ will be  torsion in $\CH^1(R)$ if $[M]=0$ in $\overline G(R)_{\mathbb Q}$,
a possibly weaker but somewhat less technical version would be:

\begin{conj}\label{genGabber}
Let $R$ be local complete intersection of dimension $3$. Let $N$ be a reflexive $R$-module which is  locally free of constant rank  on $U_R$. Furthermore, assume that $[N]=0$ in $\overline G(R)_{\mathbb Q}$, the reduced Grothendieck group of $R$ with rational coefficients. Then $\Hom_R(N,N)$ is a maximal Cohen-Macaulay $R$-module if and only if $N$ is free.  
\end{conj} 

In view of the proof of the key Proposition \ref{mainProp} and previously known results for regular and hypersurface rings, we feel it is reasonable to make:

\begin{conj}\label{ci_rigid}
Let $R$ be local complete intersection (of arbitrary dimension). Let $M,N$ be $R$-modules such that $M$ is  locally free of constant rank  on $U_R$ and $[N]=0$ in $\overline G(R)_{\mathbb Q}$. Then $(M,N)$ is $\Tor$-rigid, in the sense that for any $i>0$, $\Tor_i^R(M,N)=0$ forces $\Tor_j^R(M,N)=0$ for $j\geq i$. 
\end{conj}

By going through the proofs of \ref{mainProp} and \ref{mainCor} one can see easily that an affirmative answer to Conjecture \ref{ci_rigid} (in dimension $3$) would imply Conjecture \ref{genGabber}.

$\Tor$-rigidity has been a subject of active investigation in commutative algebra. For more in-depth discussion and references, we refer to the introduction of \cite{Da1} and the bibliography there. It is well-known that if $R$ is regular then $\Tor$-rigidity holds for any pair of modules by work of Auslander and Lichtenbaum \cite{Au, Lich}. Furthermore, Conjecture \ref{ci_rigid} is known when $\hat R$ is a hypersurface in an equicharacteristic or unramified regular local ring, see \cite{Da1, Da2}. A simple unknown case is when  $M,N$ are  $0$-dimensional, in such situation the conditions on $M$ and $N$ are automatic, so the above Conjecture would just say that any pair of finitely generated $R$-modules of finite length over a local complete intersection  is $\Tor$-rigid. The case when one of the modules has finite length is discussed in the last section of \cite{Da2}, and is still unknown to the best of our knowledge.

\end{document}